\documentclass[a4paper]{amsart}

\usepackage[utf8]{inputenc}
\usepackage[english]{babel}
\usepackage[final, colorlinks]{hyperref}

\usepackage{amsmath, amscd, amsfonts, amsthm, amssymb, bbm, mathtools}
\usepackage{float, graphicx, tikz}
\usepackage{apptools}
\usepackage{a4wide}
\usepackage{comment}

\title{The Expected Number of Roots over The Field of $p$-adic Numbers}
\author{Roy Shmueli}

\address{Raymond and Beverly Sackler School of Mathematical Sciences, Tel Aviv University, Tel Aviv 69978, Israel}
\email{royshmueli@tauex.tau.ac.il}

\AtAppendix{\counterwithin{thm}{section}}
\AtAppendix{\counterwithin{equation}{section}}

\theoremstyle{plain}
\newtheorem*{thm*}{Theorem}
\newtheorem{thm}{Theorem}
\newtheorem{cor}[thm]{Corollary}
\newtheorem{prop}[thm]{Proposition}

\newtheorem{lem}[thm]{Lemma}

\theoremstyle{definition}
\newtheorem{defn}[thm]{\protect{Definition}}

\theoremstyle{remark}

\makeatletter

\newcommand*{\ZZ}{\mathbb{Z}}
\newcommand*{\QQ}{\mathbb{Q}}
\newcommand*{\RR}{\mathbb{R}}
\newcommand*{\CC}{\mathbb{C}}
\newcommand*{\FF}{\mathbb{F}}

\renewcommand*{\Pr}{\mathbb{P}}
\newcommand*{\Ex}{\mathbb{E}}
\newcommand*{\cond}{\middle|}

\newcommand{\diff}{\!\mathop{}\mathrm{d}}

\DeclarePairedDelimiter{\pa}{\lparen}{\rparen}
\DeclarePairedDelimiter{\br}{\lbrack}{\rbrack}
\DeclarePairedDelimiter{\ang}{\langle}{\rangle}
\DeclarePairedDelimiter{\set}{\{}{\}}
\DeclarePairedDelimiter{\abs}{\lvert}{\rvert}
\DeclarePairedDelimiter{\floor}{\lfloor}{\rfloor}

\newcommand{\bigcupdot}{\mathop{%
    \vphantom{\bigcup}%
    \mathpalette\@bigcupdot{}%
}}

\newcommand*{\@bigcupdot}[2]{%
  \ooalign{%
    $\m@th#1\bigcup$\cr
    \sbox0{$#1\bigcup$}%
    \dimen@=\ht0 %
    \advance\dimen@ by -\dp0 %
    \sbox0{\scalebox{2}{$\m@th#1\cdot$}}%
    \advance\dimen@ by -\ht0 %
    \dimen@=.5\dimen@
    \hidewidth\raise\dimen@\box0\hidewidth
  }%
}

\renewcommand{\pod}[1]{\allowbreak\mathchoice
  {\if@display \mkern 18mu\else \mkern 8mu\fi (#1)}
  {\if@display \mkern 18mu\else \mkern 8mu\fi (#1)}
  {\mkern4mu(#1)}
  {\mkern4mu(#1)}
}

\newcommand*{\dlt}{\varepsilon_1}

\newcommand*{\vecz}{\vec{0}}
\newcommand*{\step}{\vec{v}}

\renewcommand*{\u}{\vec{u}}
\newcommand*{\w}{\vec{w}}
\newcommand*{\x}{\vec{x}}
\newcommand*{\y}{\vec{y}}

\newcommand*{\rC}{\xi}

\newcommand*{\xv}{\vec{X}}
\newcommand*{\alv}{\vec{\alpha}}
\newcommand*{\rCv}{\vec{\rC}}

\newcommand*{\rCount}[1]{C\pa*{{#1}}}
\newcommand*{\rCountF}[2][\QQ_{p}]{C_{#1}\pa*{{#2}}}

\newcommand*{\hCount}[2][k]{H_{#1}\pa*{{#2}}}
\newcommand*{\htCount}[2][k]{H'_{#1}\pa*{{#2}}}

\newcommand*{\Rmod}[1]{\ZZ/{#1}\ZZ}

\newcommand*{\cAnd}[2]{\left\{  \begin{array}{l}
#1\\
#2
\end{array}\right.}
\newcommand*{\cAndA}[3]{\left\{  \begin{array}{l}
#1\\
#2\\
#3
\end{array}\right.}

\newcommand*{\lift}[1]{{L\pa*{{#1}}}}

\DeclareMathOperator{\den}{Weight}

\hypersetup{
  pdftitle={The Expected Number of Roots over The Field of p-adic Numbers},
  pdfauthor={\@author}
}

\makeatother

\begin{document}

\begin{abstract}
We study the roots of a random polynomial over the field of $p$-adic numbers.
For a random monic polynomial with i.i.d.\ coefficients in $\ZZ_p$, we obtain an estimate for the expected number of roots of this polynomial. In particular,
if the coefficients take the values $\pm1$ with equal probability, the expected number of $p$-adic roots converges to
$\pa*{p-1}/\pa*{p+1}$ as the degree of the polynomial tends to $\infty$.
\end{abstract}

\maketitle

\section{Introduction}
\label{sec:introduction}

Consider the random polynomial
\[
f\pa*{X}=\rC_{0}+\rC_{1}X+\dots+\rC_{n-1}X^{n-1}+X^{n}
\]
where $\rC_{0},\dots,\rC_{n-1}$ are i.i.d.\ random variables taking values in a field $F$.
We denote by $\rCountF[F]f$ the number of roots of $f$ in $F$ without multiplicities, i.e.\
\begin{equation}
C_{F}\pa*{f}=\#\set*{ x\in F:f\pa*{x}=0} \text{.}\label{eq:root-count}
\end{equation}
We ask the question: What is the expected value of $\rCountF[F]f$
for fields $F$ of interest?

This question goes back to Bloch and P\'olya \cite{bloch_roots_1932} who showed that when $\rC_0,\dots,\rC_{n-1}$ are Rademacher random variables, i.e.\ taking the values $\pm 1$ with equal probabilities, then $\Ex\br*{\rCountF[\RR]{f}}=O\pa*{\sqrt{n}}$ as $n\to\infty$.
Since then this question has been studied extensively for $F=\RR$.
Littlewood and Offord \cite{littlewood_number_1938} improved Bloch and P\'olya \cite{bloch_roots_1932} bound and showed the improved bound on two more distributions of $\rC_i$: standard Gaussian distribution, and uniform distribution on the interval $\br*{-1, 1}$.
The first asymptotic formula was obtained by Kac \cite{kac_average_1943} when $\rC_0,\dots,\rC_{n-1}$ are standard Gaussian variables.
After more than a decade, Erd\H{o}s and Offord \cite{erdos_number_1956} proved the same asymptotic formula for polynomials with Rademacher coefficients.
Their results were then generalized by Ibragimov and Maslova \cite{ibragimov_expected_1971} who showed that if $\rC_0, \dots, \rC_n$ are i.i.d.\ with $\Ex\br*{\rC_i} = 0$ and $V\br*{\rC_i}=1$, then
\[
\Ex\br*{\rCountF[\RR]{\sum_{i=0}^n \rC_i X^i}}\sim\frac{2}{\pi}\log n
\]
as $n\to\infty$.
For more recent results, see \cite{soze_real_2017a, soze_real_2017b}.

When $F=\QQ$ is the field of rational numbers, we expect to have a few roots. For example, assume $\rC_0,\dots,\rC_{n-1}$ are Rademacher random variables.
Then the only rational numbers that can be a root, in this case, are  $\pm 1$.
Moreover, we have
\begin{equation*}
\Pr\pa*{f\pa*{\pm 1}=0} = \Pr\pa*{ \pa*{\pm 1}^n+\sum_{i=0}^{n-1} \pm 1  = 0}=  O\pa*{n^{-1/2}},
\end{equation*}
so $\Ex\br*{\rCountF[\QQ]{f}} = O\pa*{n^{-1/2}}$. This argument may be generalized to other coefficients using  the Rogozin-Kolmogorv inequality \cite{rogozin_increase_1961}.

The case of a finite field has recently found applications to random polynomials over $\ZZ$.
Breuillard and Varj\'u \cite{breuillard_irreducibility_2019} settled a conjecture of Odlyzko-Poonen \cite{odlyzko_zeros_1993}, conditionally on the extended Riemann hypothesis for Dedekind zeta functions.
They proved that a random polynomial with $\pm 1$ coefficients is irreducible over $\QQ$ with probability going to $1$ as the degree goes to $\infty$. A key step in their proof is the computation of $\Ex\br*{\rCountF[\FF_p]{f}}$ for $f$ with i.i.d.\ coefficients. In particular, the following estimate may be derived from their arguments \cite[Proposition 23]{breuillard_irreducibility_2019}:
\begin{equation}
  \label{eq:int:finte-field}
  \Ex\br*{\rCountF[\FF_p]{f}} = \Pr\pa*{\rC_{0} = 0} + \frac{p-1}{p} + O\pa*{\exp\pa*{-cn}},
\end{equation}
for some $c>0$.
This result does not depend on the extended Riemann hypothesis.

This paper studies the case where $F = \QQ_p$ is the field of $p$-adic numbers.
On the one hand, $\QQ_p$ is analogous to $\RR$ since both are completions of $\QQ$ with respect to some absolute value.
On the other hand, roots in $\QQ_p$ are closely related to roots in $\FF_p$, due to Hensel's lemma.

The starting point is to consider coefficients distributing according to Haar measure on $\ZZ_p$.
Buhler, Goldestein, Moews, Rosenberg \cite{buhler_probability_2006} showed that the probability that $f$ is split over $\QQ_p$, that is, has $n$ roots in $\QQ_p$,  is $p^{-c n^2 + O\pa*{n \log n}}$.
Caruso \cite{caruso_zeros_2018} computed the expected value of the number of roots in the non-monic case.
Appendix~\ref{app:haar-random-polynomial} computes an exact formula in the monic case:
\begin{equation}
\label{eq:haar-result}
\Ex\br*{\rCountF[\QQ_p]{f}} = \frac{p}{p+1}.
\end{equation}
We used the methods of Evans \cite{evans_expected_2006} and Igusa's local zeta functions \cite{denef_report_1990}, but Caruso's \cite{caruso_zeros_2018} method might be used as well.
This was recently generalized by Bhargava, Cremona, Fisher and Gajovi\'c \cite{bhargava_density_2021} who computed all the moments of $\rCountF[\QQ_p]{f}$.
See \cite{kulkarni_padic_2021, manssour_probabilistic_2020} for more related works.

Our result deals with a rather general distribution for the coefficients. We state it in a general form and then consider specific distributions that may appear in future applications.

In this paper, a random variable taking values in $E \subseteq \QQ_p$ is a measurable function with respect to Borel $\sigma$-algebra on $\QQ_p$.
Also, we extend the definition of $\rCountF[E]{f}$ to be the number of roots of $f$ in the subset $E$ of $F$ without multiplicities, see equation \eqref{eq:root-count}.
\begin{thm}
\label{thm:main}Let $f\pa*{X}=\rC_{0}+\rC_{1}X+\dots+\rC_{n-1}X^{n-1}+X^{n}$
where $\rC_{0},\dots,\rC_{n-1}$ are i.i.d.\ random variables taking values in $\ZZ_{p}$ and distributed such that $\rC_i\bmod p$ is non-constant random variable.
Denote $f_{0}\pa*{X}=f\pa*{pX}$.
Then for any $\varepsilon>0$
\begin{equation}
	\Ex\br*{\rCountF[\QQ_{p}]f}=\Ex\br*{\rCountF[\ZZ_{p}]{f_{0}}}+\frac{p-1}{p+1}+O\pa*{n^{-1/4+\varepsilon}},\label{eq:main-result}
\end{equation}
as $n\to\infty$.
Here the implied constant depends only on $p$, $\varepsilon$ and the law of $\rC_i$.
\end{thm}

Equation \eqref{eq:main-result} is compatible with equation \eqref{eq:haar-result}.
In equation \eqref{eq:haar-result} each residue class modulo $p$ contributes $1/\pa*{p+1}$ to the number of roots (see Appendix \ref{app:haar-random-polynomial}).
In equation \eqref{eq:main-result} the non-zero residue classes modulo $p$ also  contribute $1/\pa*{p+1}$, up to an error term, in contrast the zero class contributes $\rCountF[\ZZ_p]{f_0}$.

Next we compare the $p$-adic case, i.e., equation~\eqref{eq:main-result}, with the finite field case, that is, equation~\eqref{eq:int:finte-field}.
The first term in each equation, $\Ex\br*{\rCountF[\ZZ_p]{f_0}}$ or $\Pr\pa*{\rC_i = 0}$, are the expected number of roots of $f$ which $\equiv 0 \pmod{p}$.
The second terms correspond to the number of roots in the respective fields that are $\not \equiv 0\pmod p$.
In equation \eqref{eq:int:finte-field} non-zero elements contribute to the main term $1/p$ while in equation \eqref{eq:main-result} only $1/(p+1)$.
There is a difference of roughly $1/p^2$ between the second terms due to subtle issues coming from non-simple roots.
Those same issues also cause the error term in equation \eqref{eq:main-result} to be bigger than in equation \eqref{eq:int:finte-field}.

Finally, we compare the $p$-adic and real cases.
The term $\Ex\br*{\rCountF[\ZZ_p]{f_0}}$ in equation \eqref{eq:main-result} is easy to compute in many cases.
It has the following upper bound:
\begin{equation}
	\Ex\br*{\rCountF[\ZZ_{p}]{f_{0}}}\le\frac{\Pr\pa*{\rC_{0}\equiv0\pmod{p}}}{\Pr\pa*{\rC_{0}\not\equiv0\pmod{p}}}
\text{,}\label{eq:f0-bound}
\end{equation}
see Section~\ref{sec:additional-results}. In particular, $\Ex\br*{\rCountF[\QQ_p]{f}}$ is bounded as $n\to\infty$ in contrast to $\Ex\br*{\rCountF[\RR]{f}}$.

Next, we apply Theorem~\ref{thm:main} to interesting distributions.
\begin{cor}
\label{cor:examples}
Let $f\pa*{X}=\rC_{0}+\rC_{1}X+\dots+\rC_{n-1}X^{n-1}+X^{n}$
where $\rC_{0},\dots,\rC_{n-1}$ are i.i.d. random variables taking values in $\ZZ_{p}$.
\begin{enumerate}
\item Assume each $\rC_i$ takes the values $\pm1$ each with equal probability and $p>2$. Then,
\[
\Ex\br*{\rCountF[\QQ_{p}]f}=\frac{p-1}{p+1}+O\pa*{n^{-1/4+\varepsilon}},  \qquad n\to \infty.
\]
\item Assume each $\rC_{i}$ takes the values $0$ or $1$ each with equal probability. Then,
\[
  \Ex\br*{\rCountF[\QQ_{p}]f}=\frac{3p-1}{2\pa*{p+1}}+O\pa*{n^{-1/4+\varepsilon}},\qquad n\to \infty.
\]
\item Assume each $\rC_{i}$ takes the  values $\set*{ 0,\dots,p-1}$ uniformly. Then,
\[
  \Ex\br*{\rCountF[\QQ_{p}]f}=\frac{p^{2}+1}{p\pa*{p+1}}+O\pa*{n^{-1/4+\varepsilon}},\qquad n\to \infty.
\]
\end{enumerate}
\end{cor}

Corollary~\ref{cor:examples} follows immediately from  Theorem~\ref{thm:main} and from
\begin{prop}
\label{prop:better-root-count}If $\rC_{i} = 0$ almost surely conditioned on $p \mid \rC_i$, then
\[
\Ex\br*{\rCountF[\ZZ_{p}]{f_{0}}}=\Pr\pa*{\rC_{0}=0}\text{.}
\]
\end{prop}

\subsection{Outline of the proof of Theorem~\ref{thm:main}}
From now on, we abbreviate and write
\[
	\rCount f=\rCountF[\ZZ_{p}]f.
\]
For monic $f$, we have $\rCountF[\ZZ_p] f =\rCountF[\QQ_p] f$.

A first observation is that by grouping the roots according to their value modulo $p$ we have
\begin{equation}\label{prop:root-count-split}
	\rCount f = \sum_{r=0}^{p-1}\rCount{f_{r}},
\end{equation}
where $f_{r}\pa*{X}=f\pa*{r+pX}$. So we can treat each $f_r$ separately. The case of $r=0$, gives the term $\Ex\br*{\rCount{f_0}}$.

Take $r\ne 0$ and consider the set $\Upsilon_k$ of all polynomials of the form $g\pa*{pX}\bmod p^k$, see equation~\eqref{eq:def_upsilon}. We prove that  $f_{r}\bmod{p^{k}}$ is distributed uniformly on $\Upsilon_{k}$  up to an exponentially small error (see Lemma~\ref{lem:ups-equidist}).

Applying Hensel's lemma, this gives the estimate for $\Ex\br*{\rCount{f_r}}=\frac{1}{p+1} +O\pa*{p^{-\pa*{1-\varepsilon}k/2}}$ (see Proposition~\ref{prop:ups-root-count-rephrased}). Taking $k=\Theta\pa*{\log n}$ and summing over all $r\neq 0$, complete the proof of the theorem.

Let us elaborate on the part of the uniform distribution of $f_{r}\bmod{p^{k}}$. We define a random walk on the additive group $\pa*{\Rmod{p^{k}}}^{k}$ whose $n$-th step gives the first $k$ coefficients of $f_r$, see equation~\eqref{eq:poly-taylor-coeff-vector}.
Then, we take ideas from the works of Chung, Diaconis and Graham \cite{chung_random_1987} and of Breuillard and Varj\'u \cite{breuillard_irreducibility_2019}, using Fourier analysis and convolution properties to show that the random walk "mixes" in the group, for $k = O\pa*{\log n}$  (see Proposition~\ref{prop:zero-random-walk}).

The paper is structured as follows. Section~\ref{sec:p-adic-numbers}
surveys the $p$-adic numbers.
 Section~\ref{sec:upsilon-space} introduces $\Upsilon_{k}$ and proves Proposition~\ref{prop:ups-root-count-rephrased}.
In Section~\ref{sec:random-walks}, we study the random walks in general, and in  Section~\ref{sec:fr-distribution}, we connect the random walks to polynomials modulo $p^k$.
We prove Theorem~\ref{thm:main} in Section~\ref{sec:main-theorem}.
Finally, in Section~\ref{sec:additional-results} we prove equation~\eqref{eq:f0-bound} and Proposition~\ref{prop:better-root-count}.

\begin{figure}[H]
\centering
\begin{tikzpicture}
[>=stealth, every node/.style={rectangle, draw, minimum size=0.75cm,
minimum width=2.75cm
}]
\def \h {1.2}

\node (better-root-count) at (4,0)  {Corollary~\ref{cor:examples}};

\node (main) at (2,1*\h)  {Theorem~\ref{thm:main}};
\node (zero-root-count) at (6,1*\h)  {Proposition~\ref{prop:better-root-count}};

\node (ups-equidist) at (4,2*\h)  {Lemma~\ref{lem:ups-equidist}};
\node (ups-root-count-rephrased) at (0,2*\h)  {Proposition~\ref{prop:ups-root-count-rephrased}};

\node (coeff-prob) at (4,3*\h)  {Proposition~\ref{prop:coeff-prob}};

\node (zero-random-walk) at (4,4*\h)  {Proposition~\ref{prop:zero-random-walk}};


\draw[->] (main) to (better-root-count);
\draw[->] (zero-root-count) to (better-root-count);

\draw[->] (ups-root-count-rephrased) to (main);
\draw[->] (ups-equidist) to (main);

\draw[->] (coeff-prob) to (ups-equidist);

\draw[->] (zero-random-walk) to (coeff-prob);
\end{tikzpicture}
\caption{Main Lemmas and Theorems diagram.}
\end{figure}
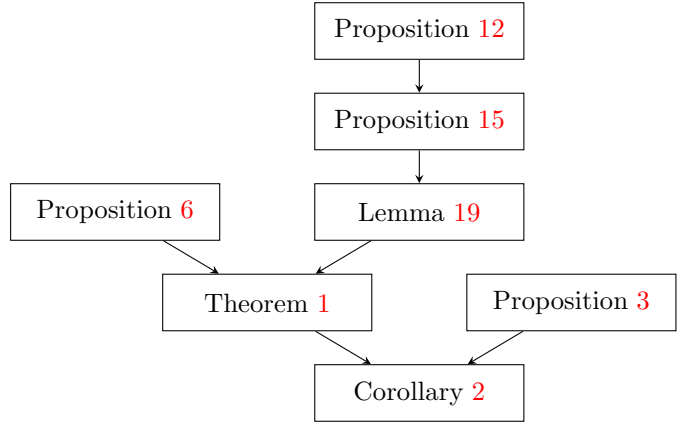

\subsection*{Acknowledgments}
I would like to thank my supervisor,  Lior Bary-Soroker, for his guidance, patience and time.
Eli Glasner for his support in the research.
Gady Kozma and  Ron Peled for their advice and comments on the research.

This research was partially supported by a grant of the Israel Science Foundation, grant no.\  702/19.
\section{The \texorpdfstring{$p$}{p}-adic numbers}
\label{sec:p-adic-numbers}

For a fixed prime number
$p$, we can write any non-zero rational number $r\in\QQ^{\times}$
as $r=p^{t}\cdot a/b$ such that $a, b, t \in \ZZ$ and $p\nmid a,b$.
We use this factorization to define \emph{the $p$-adic absolute value}:
\begin{equation*}
\abs*{r}_{p}= \begin{cases}
p^{-t}, &r\ne 0, \\
0, &r = 0.
\end{cases}
\end{equation*}
The absolute value $|\cdot|_{p}$ satisfies:
\begin{equation}
\label{eq:abs-props}
\begin{split}
	&\abs*{r}_p \ge 0 \qquad\text{and}\qquad\abs*{r}_{p} =0\iff r=0\text{,} \\
	&\abs*{r_{1}r_{2}}_{p}  =\abs*{r_{1}}_{p}\abs*{r_{2}}_{p}\text{,}\\
	&\abs*{r_{1}+r_{2}}_{p} \le\max\pa*{\abs*{r_{1}}_{p},\abs*{r_{2}}_{p}}\text{.}
\end{split}
\end{equation}

We define \emph{the field of $p$-adic numbers}, denoted by $\QQ_{p}$, as the completion of $\QQ$ with respect to $|\cdot |_p$. We define \emph{the ring of $p$-adic integers}, denoted by $\ZZ_p$, as the topological closure of $\ZZ$ in $\QQ_{p}$.
Then,
\[
  \alpha \in \ZZ_p \iff \abs*{\alpha}_p \le 1\text{.}
\]

The ring $\ZZ_p$ is local with  maximal
ideal $p\ZZ_{p}$. All the non-zero ideals are of the form
$p^{k}\ZZ_{p}$ for some integer $k\geq 0$. The quotient ring
$\ZZ_{p}/p^{k}\ZZ_{p}$ is canonically isomorphic to the ring $\ZZ/p^{k}\ZZ$.
Therefore we use the
notation of reduction modulo $p^{k}$ as in the integers, i.e., for $\alpha,\beta\in \ZZ_p$ we write
\[
\alpha\equiv\beta\pmod{p^{k}}\iff\alpha-\beta\in p^{k}\ZZ_{p}\text{.}
\]
Note that $\alpha\equiv\beta\pmod{p^{k}}\iff\abs*{\alpha-\beta}_{p}\le p^{-k}$
and that $\alpha=0\iff\alpha\equiv0\pmod{p^{k}}$ for all $k\ge 1$.

Our proof utilizes the following version of Hensel's lemma, see \cite[Theorem~II.4.2]{bachman_introduction_1964}, \cite[Proposition~II.2]{lang_algebraic_1970}, \cite[Theorem~7.3]{eisenbud_commutative_2013} or \cite[Theorem~4.1]{conrad_hensel_xxxx} for slightly weaker versions:
\begin{thm}
\label{thm:Hensel-lemma}If $f\pa*{X}\in\ZZ_{p}\br*{X}$
and $r\in\Rmod{p^{2k}}$ satisfies
\begin{equation}
\label{eq:hensel-requirements}
f\pa*{r}\equiv0\pmod{p^{2k}},\qquad f'\pa*{r}\not\equiv0\pmod{p^{k}}
\end{equation}
then $r$ can be lifted uniquely from $\Rmod{p^{k}}$ to a root of
$f$ in $\ZZ_{p}$, i.e., there is a unique $\alpha\in\ZZ_{p}$ such
that $f\pa*{\alpha}=0$ and $\alpha\equiv r\pmod{p^{k}}$.
\end{thm}
\begin{proof}
We start with proving the existence.
By abuse of notation, we denote by $r$ a lifting of $r$ to $\ZZ_p$.
Equation \eqref{eq:hensel-requirements} gives
\begin{equation*}
    \abs*{\frac{f\pa*{r}}{f'\pa*{r}^2}}_p < \frac{p^{-2k}}{\pa*{p^{-k}}^2} = 1,
\end{equation*}
and by \cite[Theorem 4.1]{conrad_hensel_xxxx} there exists a root $\alpha \in \ZZ_p$ of $f$ such that $\abs*{\alpha - r}_p = \abs*{f\pa*{r} / f'\pa*{r}}_p$.
We use equation \eqref{eq:hensel-requirements} again to infer that
\begin{equation*}
    \abs*{\alpha - r}_p = \abs*{\frac{f\pa*{r}}{f'\pa*{r}}}_p < \frac{p^{-2k}}{p^{-k}} = p^{-k}.
\end{equation*}
Thus $\alpha \equiv r \pmod{p^k}$.

To prove the uniqueness of $\alpha$, let $\alpha' \in\ZZ_p$ be another root of $f$ such that $\alpha' \equiv r \pmod{p^k}$.
We proceed by induction on $i$ to prove that $\alpha \equiv \alpha' \pmod{p^i}$.
For $i \le k$, the claim follows immediately from $\alpha' \equiv r \equiv \alpha \pmod{p^k}$.

Let $i \ge k$, we assume that $\alpha \equiv \alpha' \pmod{p^i}$ and we prove that $\alpha \equiv \alpha' \pmod{p^{i+1}}$.
So there exists $\beta \in \ZZ_p$ such that $\alpha' = \alpha + p^i \beta$.
Taylor's expansion gives
\begin{equation}
\label{eq:hensel-taylor}
    f\pa*{\alpha'} = f\pa*{\alpha + p^i \beta} = f\pa*{\alpha} + f'\pa*{\alpha} p^i \beta + \frac{1}{2} f''\pa*{\alpha} p^{2i} \beta^2 + \dots\qquad.
\end{equation}
The elements $f^{\pa{n}}\pa*{\alpha}/n!$ are all $p$-adic integers since they are the coefficients of the polynomial $f\pa*{\alpha + X} \in \ZZ_p\br*{X}$.
Hence reducing equation \eqref{eq:hensel-taylor} modulo $p^{2i}$ eliminates all the terms on the right hand side except possibly the first two, i.e.
\begin{equation}
\label{eq:hensel-taylor-modulo}
    f\pa*{\alpha'} \equiv f\pa*{\alpha} + f'\pa*{\alpha} p^i \beta \pmod{p^{2i}}.
\end{equation}
Since $f\pa*{\alpha} = f\pa*{\alpha'} = 0$, we can divide equation \eqref{eq:hensel-taylor-modulo} by $p^i$, so we get
\begin{equation*}
    0 \equiv f'\pa*{\alpha} \beta \pmod{p^{i}}.
\end{equation*}
Therefore $p \mid \beta$, otherwise $p^k \mid p^i \mid f'\pa*{\alpha}$ which contradicts $f'\pa*{\alpha} \equiv f'\pa*{r} \not\equiv 0 \pmod{p^k}$.
Thus $p^{i+1} \mid p^i \beta$ and $\alpha' = \alpha + p^i \beta \equiv \alpha \pmod{p^{i+1}}$.
\end{proof}

Further, in this paper, we use the following proposition:
\begin{prop}
\label{prop:roots-in-zp}
If $f\not\equiv0\pmod{p^{k}}$ and $\deg\pa*{f\bmod p^{k}}<m$,
then $f$ has at most $m-1$ distinct roots in $\ZZ_{p}$.
\end{prop}
\begin{proof}
We prove this proposition by contrapositive.
Let $\alpha_{1},\dots,\alpha_{m}$ be roots of $f$ in $\ZZ_{p}$ and $h\pa*{X} = \prod_{i=1}^m\pa*{X-\alpha_i}$. Dividing $f$ by $h$ with remainder in $\QQ_p[X]$, gives
\[
	f\pa*{X}=h\pa*{X}q\pa*{X}
\text{.}\]
Since $h$ is monic, $q\in \ZZ_p\br*{X}$. Reducing $f=hq$ modulo $p^k$ gives the assertion.
\end{proof}

\section{A space of polynomials modulo \texorpdfstring{$p^k$}{p\string^k}}
\label{sec:upsilon-space}

Consider the subset $\Upsilon_{k}$ of $\Rmod{p^{k}}\br*{X}$ defined by:
\begin{equation}\label{eq:def_upsilon}
	\Upsilon_{k}=\set*{ \sum_{i=0}^{k-1}c_{i}p^{i}X^{i}\in\Rmod{p^{k}}\br*{X}:\forall i<k,c_{i}\in\Rmod{p^{k-i}}} \text{.}
\end{equation}
As $\Upsilon_{k}$ is in a natural bijection with $\Rmod{p^{k}}\times\dots\times\Rmod p$,
we have
\begin{equation}
\#\Upsilon_{k}=\prod_{i=0}^{k-1}p^{k-i}=p^{k\pa*{k+1}/2}
\label{eq:ups-size-1}\text{.}
\end{equation}

\begin{prop}
\label{prop:ups-root-count-rephrased}Let $\varepsilon>0$ and let
$g\in\ZZ_{p}\br*{X}$ be a random polynomial such that $\deg g\le p^{4k}$ almost surely and such that $g\bmod{p^{k}}$ is distributed uniformly in $\Upsilon_{k}$.
Then we have
\[
\Ex\br*{\rCount g}=\frac{1}{p+1}+O\pa*{p^{-\pa*{1-\varepsilon}k/2}}
\]
as $k\to\infty$.
\end{prop}

We first introduce a definition that allows us to connect
between roots modulo $p^{k}$ and roots in $\ZZ_{p}$.
\begin{defn}
For a polynomial $g\in\ZZ_{p}\br*{X}$ we say that $x\in\Rmod{p^{k}}$
is a \emph{$k$-Henselian root of $g$} if $g'\pa*{x}\not\equiv0\pmod{p^{k}}$
and there is a lift $y$ of $x$ in $\Rmod{p^{2k}}$ such that $g\pa*{y}\equiv0\pmod{p^{2k}}$.

We say that a $k$-Henselian root $x\in\Rmod{p^{k}}$ is \emph{primitive} if $g'\pa*{x}\equiv0\pmod{p^{k-1}}$. Otherwise, we say it is \emph{non-primitive}.
\end{defn}

We denote the number of all $k$-Henselian roots of $g$ by $\hCount[k]g$.
By Hensel's lemma (Theorem~\ref{thm:Hensel-lemma}) any $k$-Henselian root can be lifted uniquely to a root in $\ZZ_{p}$, so we get that
\begin{equation}
\hCount[k]g\le C\pa*{g}
\label{eq:hensel-lemma-on-hensel-root}\text{.}
\end{equation}
We also denote
the number of all primitive $k$-Henselian roots
of $g$ by $\htCount[k]g$. So we get the following relation between them:
\begin{lem}\label{lem:primitive-hensel-roots-sum} For any polynomial $g\in\ZZ_p\br*{X}$,
\begin{equation*}
\hCount[k]g=\sum_{m=1}^{k}\htCount[m]g
\text{.}
\end{equation*}
\end{lem}
\begin{proof}
Clearly $\hCount[1]{g} = \htCount[1]{g}$ because $g'\pa*{x}\equiv0\pmod{p^{0}}$ for all $x\in\Rmod{p}$.
So, it suffices to show that
\begin{equation}
\hCount{g} = \hCount[k-1]{g} + \htCount[k]{g}
\label{eq:h-count-recursion}\text{,}
\end{equation}
and the rest follows by induction.

We write $\mathcal{H}_{k-1} \subseteq \Rmod{p^{k-1}}$ for the set of all $\pa*{k-1}$-Henselian roots of $g$, and $\widetilde{\mathcal{H}}_{k}\subseteq \Rmod{p^{k}}$ for the set of all non-primitive
$k$-Henselian roots of $g$. We define a map $\theta\colon\mathcal{H}_{k-1}\to\widetilde{\mathcal{H}}_k$ in the following manner.
For $x\in \mathcal{H}_{k-1}$, by Hensel's lemma (Theorem~\ref{thm:Hensel-lemma}) there exists a unique
lifting of $x$ to $\alpha\in \ZZ_p$ such that $g\pa*{\alpha} = 0$, so we put $\theta\pa*{x}=\alpha \bmod{p^k}$. The element $\theta\pa*{x}$ is a non-primitive Henselian root. Indeed, $g\pa*{\alpha} \equiv 0 \pmod{p^{2k}}$ and
$g'\pa*{\alpha} \not\equiv 0 \pmod{p^k}$ because $g'\pa*{\alpha} \equiv g'\pa*{x} \not\equiv 0 \pmod{p^{k-1}}$.

The map $\theta$ is injective because we have that
\[
  \theta\pa*{x} \equiv \alpha \equiv x \pmod{p^{k-1}}\text{,}
\]
meaning the reduction map modulo $p^{k-1}$ is the left inverse of $\theta$.

Moreover, the map $\theta$ is surjective. Indeed, let $y\in \widetilde{\mathcal{H}}_k$.
By Hensel's lemma (Theorem~\ref{thm:Hensel-lemma}) $y$ has a unique lift $\beta \in \ZZ_p$ such that $g\pa*{\beta} = 0$. Since $y$ is non-primitive, we have $g'\pa*{\beta}\equiv g'\pa*{y} \not\equiv 0 \pmod{p^{k-1}}$. Hence, $x:=\beta \bmod{p^{k-1}}=y\bmod{p^{k-1}}$ is a $\pa*{k-1}$-Henselian root of $g$. From the uniqueness of $\beta$
we have that $y=\theta\pa*{x}$.

Therefore, $\theta$ is a bijection. Hence the number of non-primitive $k$-Henselian roots of $g$ equals to $\hCount[k-1]{g}$, which proves equation \eqref{eq:h-count-recursion}.
\end{proof}

\begin{lem}
\label{lem:ex-k-henselian}Let $k>0$ and let $g\in\ZZ_{p}\br*{X}$
be a random polynomial such that $g\bmod{p^{2k}}$ is distributed
uniformly in $\Upsilon_{2k}$. Then
\[
\Ex\br*{\hCount[k]g}=\frac{1-p^{-2k+2}}{p+1}.
\]
\end{lem}

\begin{proof}
We start by computing $\Ex\br*{\htCount[m]{f}}$ and applying Lemma~\ref{lem:primitive-hensel-roots-sum}.
For $m=1$, since $g'\equiv 0 \pmod{p}$ by the definition of $\Upsilon_{2k}$ there are
no $1$-Henselian roots and $\Ex\br*{\htCount[1]g}=0$. For
$1<m\le k$, we write $\mathcal{H}'_{m}$ for the set of all primitive $m$-Henselian
roots of $g$, so that $\htCount[m]{g} = \#\mathcal{H}'_m$. We use the following consequence of linearity of expectation
\begin{equation}\label{eq:htm-count-technique}
\Ex\br*{\htCount[m]g}=\sum_{x\in\Rmod{p^{m}}}\Pr\pa*{x\in \mathcal{H}'_{m}}\text{.}
\end{equation}

Define $G_x$ to be the event that $g'\pa*{x}\not\equiv0\pmod{p^{m}}$
and $g'\pa*{x}\equiv0\pmod{p^{m-1}}$.
Assume $G_x$ occurs and let $y,\widetilde{y}\in \Rmod{p^{2m}}$ be two lifts of $x$.
Then
\begin{equation}\label{eq:henselian-lift-equiv}
g\pa*{y}\equiv g\pa*{\widetilde{y}}\pmod{p^{2m}}
\iff y\equiv \widetilde{y}\pmod{p^{m+1}}
\text{.}
\end{equation}
Indeed, if $y=\widetilde{y}+p^{m+1}z$ then
\begin{equation}
\label{eq:y-to-tilde-y-by-taylor}
g\pa*{y}=g\pa*{\widetilde{y}+p^{m+1}z}=g\pa*{\widetilde{y}}+g'\pa*{\widetilde{y}}p^{m+1}z+\frac{1}{2}g''\pa*{\widetilde{y}}p^{2m+2}z^{2}+\cdots\text{ .}
\end{equation}
The expressions $g^{\pa{i}}\pa*{\widetilde{y}} / i!$ are $p$-adic integers because they are the the coefficients of the polynomial $g\pa*{\widetilde{y} + X}$.
Thus, reducing equation~\eqref{eq:y-to-tilde-y-by-taylor} modulo $p^{2m}$ and using that $g'\pa*{\widetilde{y}}\equiv g'\pa*{x}\equiv0\pmod{p^{m-1}}$ gives $g\pa*{y}\equiv g\pa*{\widetilde{y}}\pmod{p^{2m}}$.
For the other direction, assume that $g\pa*{y}\equiv g\pa*{\widetilde{y}} \pmod{p^{2m}}$. Write $y=x+p^{m}z$ and $\widetilde{y}=x+p^{m}\widetilde{z}$ then
\begin{align*}
g\pa*{y}&=g\pa*{x+p^{m}z}=g\pa*{x}+g'\pa*{x}p^{m}z+\frac{1}{2}g''\pa*{x}p^{2m}z^{2}+\cdots
\text{ ,} \\
g\pa*{\widetilde{y}}&=g\pa*{x+p^{m}\widetilde{z}}=g\pa*{x}+g'\pa*{x}p^{m}\widetilde{z}+\frac{1}{2}g''\pa*{x}p^{2m}\widetilde{z}^{2}+\cdots
\text{ .}
\end{align*}
Reducing modulo $p^{2m}$, all terms in the right hand side except possibly the first two vanish.
Plugging in the assumption $g\pa*{y}\equiv g\pa*{\widetilde{y}}\pmod{p^{2m}}$,
we get that $g'\pa*{x}p^{m}z\equiv g'\pa*{x}p^{m}\widetilde{z}\pmod{p^{2m}}$.
Dividing this congruence by $p^{2m-1}$, since $g'\pa*{x}\not\equiv0\pmod{p^{m}}$
we infer that $z\equiv \widetilde{z}\pmod{p}$. Hence $y\equiv \widetilde{y}\pmod{p^{m+1}}$, as needed.

We denote by $\mathcal{L}_{x}$ the set of lifts of $x$ to $\Rmod{p^{m+1}}$.
Equation \eqref{eq:henselian-lift-equiv} means that when checking if $x$ is an $m$-Henselian
root, it suffices to check if $\mathcal{L}_x$ contains a root
of $g$ modulo $p^{2m}$.
The other direction of the equation gives us that there is at most one such root in $\mathcal{L}_x$.
Thus
\begin{equation}
\label{eq:htm-prop}
\begin{aligned}
\Pr\pa*{x\in \mathcal{H}'_{m}}
&=\Pr\pa*{\set*{\exists y\in \mathcal{L}_{x},g\pa*{y}\equiv0\pmod{p^{2m}}}\cap G_x}\\
& =\Pr\pa*{\bigcupdot_{y\in \mathcal{L}_{x}}\set*{g\pa*{y}\equiv0\pmod{p^{2m}}}\cap G_x}\\
& =\sum_{y\in \mathcal{L}_{x}}\Pr\pa*{\set*{g\pa*{y}\equiv0\pmod{p^{2m}}}\cap G_x}\text{.}
\end{aligned}
\end{equation}
We write $g\pa*{X}=\rC_{0}+\rC_{1}pX+\rC_{2}p^{2}X^{2}+\dots$\ .
Then
\begin{equation}
\label{eq:lift-prop}
\begin{aligned}
\Pr\pa*{\set*{g\pa*{y}\equiv0\pmod{p^{2m}}}\cap G_x}
&=
\Pr\pa*{\cAndA{g\pa*{y}\equiv0\pmod{p^{2m}}}{g'\pa*{y}\not\equiv0\pmod{p^{m}}}{g'\pa*{y}\equiv0\pmod{p^{m-1}}}}\\
&=\Pr\pa*{\cAndA{\rC_{0}\equiv-\pa*{\rC_{1}py+\dots}\pmod{p^{2m}}}{p\rC_{1}\not\equiv-\pa*{2\rC_{2}p^{2}y+\dots}\pmod{p^{m}}}{p\rC_{1}\equiv-\pa*{2\rC_{2}p^{2}y+\dots}\pmod{p^{m-1}}}}\\
&=\Pr\pa*{\cAndA{\rC_{0}\equiv-\pa*{\rC_{1}py+\dots}\pmod{p^{2m}}}{\rC_{1}\not\equiv-\pa*{2\rC_{2}py+\dots}\pmod{p^{m-1}}}{\rC_{1}\equiv-\pa*{2\rC_{2}py+\dots}\pmod{p^{m-2}}}}\\
&=p^{-3m}\cdot p\pa*{p-1}
\text{.}\end{aligned}
\end{equation}
The last equality holds true because the pair $\pa*{\rC_{0} \bmod p^{2m}, \rC_{1} \bmod p^{m-1}}$ is distributed
uniformly in $\Rmod{p^{2m}} \times \Rmod{p^{m-1}}$.

Finally, plugging equations \eqref{eq:htm-prop} and \eqref{eq:lift-prop} into equation \eqref{eq:htm-count-technique} gives
\[
\Ex\br*{\htCount[m]g}=\sum_{x\in\Rmod{p^{m}}}\sum_{y\in \mathcal{L}_{x}}p^{-3m}\cdot p\pa*{p-1}=p^{-2m}\cdot p^{2}\pa*{p-1}
\]
and by Lemma~\ref{lem:primitive-hensel-roots-sum} we get
\[
\Ex\br*{\hCount[k]g}=\sum_{m=2}^{k}p^{-2m}\cdot p^{2}\pa*{p-1}=\frac{1-p^{-2k+2}}{p+1}\text{.}\qedhere
\]
\end{proof}
For $x\in\Rmod{p^{k}}$ we say that $x$ is \emph{simple root
of $g$ modulo $p^{k}$} if $g\pa*{x}\equiv 0\pmod{p^{k}}$ and $g'\pa*{x}\not\equiv0\pmod{p^{k}}$.
We say $x\in\Rmod{p^{k}}$
is \emph{non-simple root of $g$ modulo $p^{k}$} if $g\pa*{x}\equiv g'\pa*{x}\equiv 0\pmod{p^{k}}$.
We denote by $M_g$ the event that $g$ has a non-simple root modulo $p^{k}$.
\begin{lem}
\label{lem:non-simple-root}For any $k>0$, let $g\in\ZZ_{p}\br*{X}$
be a random polynomial such that $g\bmod{p^{k}}$ is distributed uniformly
in $\Upsilon_{k}$. Then we have that
\[
\Pr\pa*{M_g}\le p^{-k+1}.
\]
\end{lem}

\begin{proof}
Let $\mathcal{M}$ be the set of all non-simple roots of $g$ modulo $p^{k}$.
We write $g\pa*{X}=\rC_{0}+\rC_{1}pX+\rC_{2}p^{2}X^{2}+\dots$
and then for a fixed $x\in\Rmod{p^k}$ we have
\begin{equation}
\label{eq:single-non-simple-prob}
\begin{aligned}
\Pr\pa*{x\in \mathcal{M}} & =\Pr\pa*{\cAnd{g\pa*{x}\equiv0\pmod{p^{k}}}{g'\pa*{x}\equiv0\pmod{p^{k}}}}\\
 & =\Pr\pa*{\cAnd{\rC_{0}\equiv-\pa*{\rC_{1}px+\dots}\pmod{p^{k}}}{p\rC_{1}\equiv-\pa*{2\rC_{2}p^{2}x+\dots}\pmod{p^{k}}}}\\
 & =\Pr\pa*{\cAnd{\rC_{0}\equiv-\pa*{\rC_{1}px+\dots}\pmod{p^{k}}}{\rC_{1}\equiv-\pa*{2\rC_{2}px+\dots}\pmod{p^{k-1}}}}\\
 & =p^{-2k+1}\text{.}
\end{aligned}
\end{equation}
The last equality holds true because $\pa*{\rC_{0} \bmod p^{2m}, \rC_{1} \bmod p^{m-1}}$ are distributed
uniformly in $\Rmod{p^{2m}} \times \Rmod{p^{m-1}}$.

We finish the proof by using union bound and plugging equation~\eqref{eq:single-non-simple-prob} obtaining that
\begin{equation*}
  \Pr\pa*{M_g}
  = \Pr\pa*{\bigcup_{x\in\Rmod{p^k}} \set*{x\in\mathcal{M}}}
  \le \sum_{x\in\Rmod{p^k}} \Pr\pa*{x\in\mathcal{M}}= p^{-k+1}
  \text{.} \qedhere
\end{equation*}
\end{proof}

\begin{prop}
\label{prop:ex-ups-root-count}Let $\varepsilon>0$ and let $g\in\ZZ_{p}\br*{X}$ be
a random polynomial such that $\deg g\le p^{9k}$ almost surely and
$g\bmod{p^{2k}}$ is distributed uniformly in $\Upsilon_{2k}$. Then
\[
\Ex\br*{\rCount g}=\frac{1}{p+1}+O\pa*{p^{-\pa*{1-\varepsilon}k}}
\]
as $k\to\infty$.
\end{prop}
\begin{proof}
By Hensel's lemma (Theorem~\ref{thm:Hensel-lemma}), we know that any $k$-Henselian root lifts uniquely
to a root in $\ZZ_{p}$. Moreover, if $x$ is a simple root of $g$
modulo $p^{k}$ then $x$ lifts to a root in $\ZZ_{p}$ if and only if
$x$ is $k$-Henselian. Indeed, if $x$ lifts to a root
$\alpha \in \ZZ_{p}$, the $\alpha\bmod{p^{2k}}$ is a lift of $x$ to a root of $g$ in $\Rmod{p^{2k}}$.

The number of roots of $g$ in $\ZZ_p$ that reduce to a non-simple root modulo $p^k$ is at most $\deg g$, in particular when $g\equiv0\pmod{p^{k}}$.
When $g\not\equiv0\pmod{p^{k}}$, Proposition~\ref{prop:roots-in-zp} bounds that number by $\deg\pa*{g \bmod p^k} < k$.

This yields the following upper bound on the expected number of roots of $g$,
\begin{equation}\label{eq:root-count-upper-estimate}
\Ex\br*{H_{k}\pa*{g}} \le\Ex\br*{C\pa*{g}}\le\Ex\br*{H_{k}\pa*{g}}+k\Pr\pa*{M_g}
+\Ex\br*{\deg g}\Pr\pa*{g\equiv0\pmod{p^{k}}}
\text{.}
\end{equation}
By Lemma~\ref{lem:ex-k-henselian} and Lemma~\ref{lem:non-simple-root}
\begin{align}
\Ex\br*{H_{k}\pa*{g}} &=\frac{1}{p+1}+O\pa*{p^{-2k}}
\label{eq:hk-estimate}\text{,}\\
k\Pr\pa*{M_g}&=kO\pa*{p^{-k+1}}=O\pa*{p^{-\pa*{1-\varepsilon}k}}
\label{eq:mk-estimate}\text{.}
\end{align}
Finally, since $\deg g\le p^{9k}$ almost surely and by equation \eqref{eq:ups-size-1}, we get that
\begin{equation}\label{eq:deg-zero-estimate}
\Ex\br*{\deg g}\Pr\pa*{g\equiv0\pmod{p^{k}}}=O\pa*{p^{9k}}p^{-k\pa*{k+1}/2}=O\pa*{p^{-k^{2}/4}}\text{.}
\end{equation}
Plugging equations \eqref{eq:hk-estimate}, \eqref{eq:mk-estimate} and \eqref{eq:deg-zero-estimate} into equation \eqref{eq:root-count-upper-estimate}
finishes the proof.
\end{proof}
\begin{proof}[Proof of Proposition~\ref{prop:ups-root-count-rephrased}]
If $k$ is even we just apply Proposition~\ref{prop:ex-ups-root-count} substituting $k$ by $k/2$.
Otherwise $k$ is odd. Since $g\bmod{p^k}$ is distributed uniformly in $\Upsilon_{k}$ then $g\bmod{p^{k-1}}$ is distributed uniformly in
$\Upsilon_{k-1}$. Moreover, $\deg g \le p^{4k} = p^{8\pa*{k-1}/2+4} \le p^{9\pa*{k-1}/2}$ almost surely for $k$ sufficiently large.
So we apply Proposition~\ref{prop:ex-ups-root-count} again substituting $k$ by $k/2$ to finish the proof.
\end{proof}

\section{Random walks on \texorpdfstring{$\pa*{\Rmod q}^{d}$}{(Z/qZ)\string^d}}
\label{sec:random-walks}

Let $q$ be a powers of the prime number $p$ and $V=\pa*{\Rmod q}^{d}$.
Let $\rC_{0},\rC_{1},\dots,\rC_{n-1}$ be i.i.d.\ random variables
taking values in $\ZZ_p$ distributed according to a law $\mu$.
We choose some vectors $\step_{0},\step_{1},\dots,\step_{n-1}$ in $V$.
For some $r\in\pa*{\Rmod q}^{\times}$
we study the random walk over the additive group $\pa*{V,+}$
whose $n$-th step is $\sum_{i=0}^{n-1}\rC_{i}r^{i}\step_{i}$. We denote
by $\nu_{r}$ the probability measure induced from the
$n$-th step.

For two vectors $\u,\w\in V$, we denote by
$\ang*{ \u,\w} $ the formal dot product, i.e.
\[
\ang*{ \u,\w} =u_{1}w_{1}+\dots+u_{d}w_{d}\text{.}
\]
For a non-zero vector $\u\in V$,
we call the number of vectors in $\step_{0},\dots,\step_{n-1}$ such that
$\ang*{ \u, \step_i} \ne0$, the \emph{$\u$-weight of $\step_{0},\dots,\step_{n-1}$}
and we denote it by $\den_{\u}\pa*{\step_{0},\dots,\step_{n-1}}$.
We define the \emph{minimal weight of $\step_{0},\dots,\step_{n-1}$}
to be
\[
\sigma\pa*{\step_{0},\dots,\step_{n-1}}=\min_{\u\in V\setminus\set*{\vecz}
}\den_{\u}\pa*{\step_{0},\dots,\step_{n-1}}\text{.}
\]

We define $\tau$ to be the number
\begin{equation}
\tau=1-\sum_{x\in\Rmod p}\mu\pa*{x+p\ZZ_{p}}^{2}
\label{eq:tau}\text{.}
\end{equation}
Note that $\tau \ge 0$ since $\mu\pa*{E}^2 \le \mu\pa*{E} \le 1$ and $\sum_{x\in\Rmod p}\mu\pa*{x+p\ZZ_{p}}=1$.

The relation between $\tau$, $\sigma$ and the measure $\nu_{r}$
is found in the following proposition, cf.\ \cite[Proposition 23]{breuillard_irreducibility_2019}.
\begin{prop}
\label{prop:zero-random-walk}
For any $r\in\pa*{\Rmod q}^{\times}$
and $\u\in V$, we have
\[
\nu_{r}\pa*{\u}=\frac{1}{\#V}+O\pa*{\exp\pa*{-\frac{\tau\sigma\pa*{\step_{0},\step_{1},\dots,\step_{n-1}}}{q^{2}}}}
\text{.}\]

\end{prop}

Let $\mu_{q}$ and $\mu_{p}$ be the pushforward of $\mu$ to $\Rmod{q}$ and $\Rmod{p}$ respectively. Those measures satisfy the following
\[
\mu_{q}\pa*{x}=\mu\pa*{x+q\ZZ_{p}}
\qquad\text{and}\qquad
\mu_{p}\pa*{x}=\mu\pa*{x+p\ZZ_{p}}
\text{.}\]
We can use this notation to write $\tau=1-\sum_{x\in\Rmod p}\mu_p\pa*{x}^{2}$.

Let $\delta_{\w}$ be the Dirac measure on $V$, i.e.
\begin{equation}
  \delta_{\w}\pa*{\u} = \begin{cases}
    1, &\u = \w ,\\
    0, &\u \ne \w.
  \end{cases}
\label{eq:dirac-delta}
\end{equation}
We write $\mu.\delta_{\w}$ for the following probability measure on $V$:
\begin{equation}
\mu.\delta_{\w}\pa*{\cdot}=\sum_{x\in\Rmod q}\mu_{q}\pa*{x}\delta_{x\w}\pa*{\cdot}
\text{.}\label{eq:small-conv}
\end{equation}
With this notation, we can write:
\begin{equation}
\nu_{r}=\mu.\delta_{\step_{0}}\ast\mu.\delta_{\step_{1}r}\ast\dots\ast\mu.\delta_{\step_{n-1}r^{n-1}}
\label{eq:random-walk-conv}\text{.}
\end{equation}
where $*$ is the convolution operator.

In this section we denote the Fourier transform by $\widehat{\cdot}$
and we let $\zeta$ be a primitive $q$-th root of unity.
So for any function $f\colon V \to \CC$ we have the following relations
\begin{align}
\hat{f}\pa*{\u}&=\sum_{\w\in V}f\pa*{\w}\zeta^{-\ang*{ \u,\w} }\quad\text{and}\label{eq:fourier-trasform} \\ f\pa*{\u}&=\frac{1}{\#V}\sum_{\w\in V}\hat{f}\pa*{\w}\zeta^{\ang*{ \u,\w} }\text{.}\label{eq:inv-fourier-trasform}
\end{align}

The following lemma and its proof are based on \cite[lemma 31]{breuillard_irreducibility_2019}.
\begin{lem}
\label{lem:1-step-fourier-bound}Let $\u,\w\in V$. If $\ang*{ \u,\w} \ne0$
then
\[
\abs*{\widehat{\mu.\delta_{\w}}\pa*{\u}}\le\exp\pa*{-\frac{\tau}{q^2}}.
\]
\end{lem}

\begin{proof}
By direct computation using equations \eqref{eq:fourier-trasform} and \eqref{eq:small-conv} we get
\begin{align*}
\abs*{\widehat{\mu.\delta_{\w}}\pa*{\u}}^{2} & =\sum_{\x,\y\in V}\mu.\delta_{\w}\pa*{\x}\mu.\delta_{\w}\pa*{\y}\zeta^{\ang*{ \x-\y,\u} }\\
 & =\sum_{\x,\y\in V}\sum_{x,y\in\Rmod q}\mu_{q}\pa*{x}\delta_{x\w}\pa*{\x}\mu_{q}\pa*{y}\delta_{y\w}\pa*{\y}\zeta^{\ang*{ \x-\y,\u} }\text{.}
\end{align*}
Then from equation \eqref{eq:dirac-delta}
\[
  \abs*{\widehat{\mu.\delta_{\w}}\pa*{\u}}^{2} = \sum_{x,y\in\Rmod q}\mu_{q}\pa*{x}\mu_{q}\pa*{y}\zeta^{\pa*{x-y}\ang*{ \w,\u} }\text{.}
\]
We denote by $\lift t$ the lift of $t \in \Rmod q$
to the interval $\left(-\frac{q}{2},\frac{q}{2}\right]\cap\ZZ$.
Since $\abs*{\widehat{\mu.\delta_{\w}}\pa*{\u}}^{2}\in\RR$ and
$\Re\pa*{\zeta^{t}}\le1-2\lift{t}^{2}/q^{2}$, we get
\begin{align*}
\abs*{\widehat{\mu.\delta_{\w}}\pa*{\u}}^{2} & =\sum_{x,y\in\Rmod q}\mu_{q}\pa*{x}\mu_{q}\pa*{y}\Re\pa*{\zeta^{\pa*{x-y}\ang*{\w,\u} }}\\
 & \le\sum_{x,y\in\Rmod q}\mu_{q}\pa*{x}\mu_{q}\pa*{y}\pa*{1-\frac{2\lift{\pa*{x-y}\ang*{\w,\u}}^{2}}{q^{2}}}\\
 & =1-\frac{2}{q^2}\sum_{x,y\in\Rmod q}\mu_{q}\pa*{x}\mu_{q}\pa*{y}\lift{\pa*{x-y}\ang*{ \w,\u}}^{2}
 \text{.}
\end{align*}
If $p\nmid x-y$ then $\pa*{x-y}\ang*{ \w,\u}$ is non-zero. So
\begin{equation*}
\lift{\pa*{x-y}\ang*{ \w,\u}}^2 \ge 1
\text{,}
\end{equation*}
hence
\begin{equation}
\abs*{\widehat{\mu.\delta_{\w}}\pa*{\u}}^{2}
\le1-\frac{2}{q^{2}}\sum_{\substack{x,y\in\Rmod q\\
p\nmid x-y
}
}\mu_{q}\pa*{x}\mu_{q}\pa*{y}
\label{eq:1-step-fourier-bound-1}\text{.}
\end{equation}

Since $\mu_p$ is also the pushforward measure of $\mu_q$, we have
\[
\mu_p \pa*{x'} = \sum_{\substack{x\in\Rmod{q} \\ x'\equiv x\pmod{p}}} \mu_q\pa*{x}
\text{.}\]
Hence
\[
\sum_{\substack{x,y\in\Rmod q\\p\nmid x-y}
}\mu_{q}\pa*{x}\mu_{q}\pa*{y} =\sum_{\substack{x',y'\in\Rmod p\\
x'\ne y'
}
}\mu_p\pa*{x'}\mu_p\pa*{y'}
\text{.}\]
By direct computation
\begin{align*}
\sum_{\substack{x,y\in\Rmod q\\
p\nmid x-y
}
}\mu_{q}\pa*{x}\mu_{q}\pa*{y} & =\sum_{x',y'\in\Rmod p}\mu_p\pa*{x'}\mu_p\pa*{y'}-\sum_{x'\in\Rmod p}\mu_p\pa*{x'}^{2}\\
 & =\pa*{\sum_{x'\in\Rmod p}\mu_p\pa*{x'}}^2-\sum_{x'\in\Rmod p}\mu_p\pa*{x'}^{2}\\
  & =1-\sum_{x'\in\Rmod p}\mu_p\pa*{x'}^{2}=\tau
\text{.}
\end{align*}
Plugging this into equation \eqref{eq:1-step-fourier-bound-1} and using the inequality $1-t\le\exp\pa*{-t}$, we get
\begin{align*}
\abs*{\widehat{\mu.\delta_{\w}}\pa*{\u}}^2
  \le 1-\frac{2\tau}{q^{2}} \le \exp\pa*{-\frac{2\tau}{q^{2}}}
\text{.}
\end{align*}
We finish the proof by taking square root on both sides of the inequality.
\end{proof}
\begin{lem}
\label{lem:fourier-bound} Let $r\in\pa*{\Rmod q}^{\times}$
and $\u\in V\setminus\set*{ \vecz} $. Then
\[
\abs*{\hat{\nu}_{r}\pa*{\u}}\le\exp\pa*{-\frac{\tau\den_{\u}\pa*{\step_{0},\step_{1},\dots,\step_{n-1}}}{q^{2}}}
\text{.}
\]
\end{lem}

\begin{proof}
We define the following set $I\pa*{\u}=\set*{ 0\le i<n:\ang*{ \u,\step_{i}} \ne0} $, so that $$
\den_{\u}\pa*{\step_{0},\step_{1},\dots,\step_{n-1}}=\#I\pa*{\u}
$$
by definition.
For $i\in I\pa*{\u}$, Lemma~\ref{lem:1-step-fourier-bound}
infers that
\begin{equation}
\abs*{\widehat{\mu.\delta_{\step_{i}r^{i}}}\pa*{\u}}\le\exp\pa*{-\frac{\tau}{q^{2}}}
\label{eq:bound-step-in-set}\text{.}
\end{equation}

Otherwise, for $i\notin I\pa*{\u}$ we have that
\begin{equation}
\abs*{\widehat{\mu.\delta_{\step_{i}r^{i}}}\pa*{\u}}\le\sum_{\w\in V}\abs*{\mu.\delta_{\step_{i}r^{i}}\pa*{\w}\zeta^{-\ang*{ \u,\w} }}=1
\label{eq:bound-step-not-in-set}\text{.}
\end{equation}
By equations \eqref{eq:random-walk-conv}, \eqref{eq:bound-step-in-set}, \eqref{eq:bound-step-not-in-set} and since the Fourier transform maps convolutions to products we get
\begin{align*}
\abs*{\hat{\nu}_{r}\pa*{\u}} & =\abs*{\prod_{i=0}^{n-1}\widehat{\mu.\delta_{\step_{i}r^{i}}}\pa*{\u}}\\
 & \le\abs*{\prod_{i\in I\pa*{\u}}\exp\pa*{-\frac{\tau}{q^{2}}}}\cdot\abs*{\prod_{i\notin I\pa*{\u}}1}\\
 & =\exp\pa*{-\frac{\tau\#I\pa*{\u}}{q^{2}}} \\
 & =\exp\pa*{-\frac{\tau\den_{\u}\pa*{\step_{0},\step_{1},\dots,\step_{n-1}}}{q^{2}}}
\text{.}\qedhere
\end{align*}
\end{proof}
\begin{proof}[Proof of Proposition~\ref{prop:zero-random-walk}]
For any probability measure $\nu$ on $V$, we have by equation \eqref{eq:fourier-trasform} that
\[
\hat{\nu}\pa*{\vecz} = \sum_{\w \in V} \nu\pa*{\w} = 1\text{.}
\]
Hence, by equation \eqref{eq:inv-fourier-trasform}
\begin{equation}
\nu\pa*{\vecz}=\frac{1}{\#V}\sum_{\w\in V}\hat{\nu}\pa*{\w}\zeta^{\ang*{ \vecz,\w} }=\frac{1}{\#V}+\frac{1}{\#V}\sum_{\w\in V\setminus\set*{ \vecz} }\hat{\nu}\pa*{\w}
\label{eq:prop-of-zero}\text{.}
\end{equation}

Since we have that  $\nu_{r}\pa*{\u}=\pa*{\delta_{-\u}\ast\nu_{r}}\pa*{\vecz}$ and $\hat{\delta}_{-\u}\pa*{\cdot} = \zeta^{\ang*{\cdot, \u}}$, by plugging it into equation \eqref{eq:prop-of-zero} we get
\[
\nu_{r}\pa*{\u}
  =\pa*{\delta_{-\u}\ast\nu_{r}}\pa*{\vecz}
  =\frac{1}{\#V}+\frac{1}{\#V}\sum_{\w\in V\setminus\set*{ \vecz} }\hat{\nu}_{r}\pa*{\w}\zeta^{\ang*{\w, \u}}
  \text{.}
\]
Therefore, by the triangle inequality and Lemma~\ref{lem:fourier-bound}
\begin{align*}
\abs*{\nu_{r}\pa*{\u}-\frac{1}{\#V}} & \le\frac{1}{\#V}\sum_{\w\in V\setminus\set*{ \vecz} }\abs*{\hat{\nu}_{r}\pa*{\w}}\\
 & \le\frac{1}{\#V}\sum_{\w\in V\setminus\set*{ \vecz} }\exp\pa*{-\frac{\tau\sigma\pa*{\step_{0},\step_{1},\dots,\step_{n-1}}}{q^{2}}}\\
 & <\exp\pa*{-\frac{\tau\sigma\pa*{\step_{0},\step_{1},\dots,\step_{n-1}}}{q^{2}}}\text{.}\qedhere
\end{align*}
\end{proof}
\section{The distribution of \texorpdfstring{${f^{\pa*{i}}\pa*{r}/i!}$}{f\string^(i)(r)/i!} modulo powers of \texorpdfstring{$p$}{p}}
\label{sec:fr-distribution}

In this section we use Proposition~\ref{prop:zero-random-walk} to
find the distribution of the Taylor coefficients of $f$, ${f^{\pa*{i}}\pa*{r}/i!}$ modulo a power of $p$.
Note it is possible to talk about $f^{\pa{i}}\pa*{r}/i!$ modulo a power of $p$ since those terms are the coefficients of the polynomial $f\pa*{r+X}$ hence $p$-adic integers.

\begin{prop}
\label{prop:coeff-prob}Let $f$ be random polynomial defined as in
Theorem~\ref{thm:main} and let $d<n$ be a positive integer.
Also, let $m_{0},\dots,m_{d-1}$ be non-negative integers and $\gamma_{0},\dots,\gamma_{d-1}\in\ZZ_{p}$ be $p$-adic integers.
There exists $\tau > 0$ depending only on the distribution of $\rC_i$, such that for any integer $1\le r<p$,
\begin{multline*}
\Pr\pa*{\forall i\in\set*{ 0,\dots,d-1} ,\frac{1}{i!}f^{\pa*{i}}\pa*{r}\equiv\gamma_{i}\pmod{p^{m_{i}}}}\\
=p^{-N}+O\pa*{\exp\pa*{-\frac{\tau n}{p^{2M+1}d}+dM\cdot\log p}}
\end{multline*}
where $M=\max_{0\le i<d}m_{i}$ and $N=\sum_{i=0}^{d-1}m_{i}$.
\end{prop}

We shall need the following three auxiliary results before going to the proof.
The first lemma is a consequence of Lucas's theorem (see \cite{fine_binomial_1947}).
\begin{lem}
\label{lem:binom-cycle}
Let $p$ be a prime and $\ell$ be a positive integer.
Then for any non-negative integers $m,n$ such that $m<p^{\ell}$ we have
\[
\binom{n+p^{\ell}}{m}\equiv\binom{n}{m}\pmod{p}.
\]
\end{lem}

Consider the vectors in $\step_{0},\dots,\step_n\in\pa*{\Rmod{p^k}}^{d}$
such that
\begin{equation}\label{eq:pascal-vector}
\step_{i}=\pa*{\binom{i}{0}\bmod p^k,\dots,\binom{i}{d-1}\bmod p^k}\text{.}
\end{equation}
where we define $\binom{i}{j}=0$ for $i<j$. We call those vectors
the \emph{Pascal vectors of length $d$ modulo $p^k$}. We are interested
in finding a lower bound for the minimal weight of the Pascal vectors
of length $d$ modulo $p^k$.

\begin{lem}
\label{lem:minimal-den-mod-p}Let $n\ge d\ge1$ be integers, let $p$ be a
prime and let $\step_{0},\dots,\step_{n-1}$ be the Pascal vectors of
length $d$ modulo $p$. Then we have that
\[
\frac{n}{pd} - 1\le\sigma\pa*{\step_{0},\dots,\step_{n-1}}.
\]
\end{lem}

\begin{proof}
Let $\ell$ be the integer such that $p^{\ell-1} < d \le p^\ell$. 
The first $d$ vectors $\step_{0},\dots,\step_{d-1}$ are forming
a basis in $\pa*{\Rmod {p}}^{d}$ since
\[
\pa*{\begin{array}{ccc}
- & \step_{0} & -\\
- & \step_{1} & -\\
 & \vdots\\
- & \step_{d-1} & -
\end{array}}=\pa*{\begin{array}{ccccc}
1 & 0 & 0 & \cdots & 0\\
1 & 1 & 0 & \cdots & 0\\
1 & 2 & 1 & \cdots & 0\\
\vdots & \vdots & \vdots & \ddots & \vdots\\
1 & d-1 & \binom{d-1}{2} & \cdots & 1
\end{array}}\bmod{p} \text{.}
\]
Given $\u\in\pa*{\Rmod {p}}^{d}$
be a non-zero vector, we have some $i_{0}<d$ such that $\ang*{ \u,\step_{i_{0}}} \ne0$.
So
by Lemma~\ref{lem:binom-cycle} we get that $\step_{i}=\step_{i+p^{\ell}}$
and in particular $\ang*{\u,\step_{i_{0}+m p^{\ell}}} \ne0$
for all $m\le\pa*{n-i_{0}-1}/p^{\ell}$. So we found $\floor*{\pa*{n-i_{0}-1}/p^{\ell}} +1$
vectors $\step_{i}$ such that $\ang*{\u,\step_{i}} \ne0$.
Hence,
\begin{align*}
\den_{\u}\pa*{\step_{0},\dots,\step_{n-1}} & \ge\floor*{ \frac{n-i_{0}-1}{p^{\ell}}} +1\\
 & > \frac{n}{p^{\ell}} - \frac{i_{0}+1}{p^{\ell}}\\
 & \ge \frac{n}{pd} - 1\text{.}\qedhere
\end{align*}
\end{proof}
\begin{cor}
\label{cor:minimal-den-mod-q}
Let $\step_{0},\dots,\step_{n-1}$ be the Pascal vectors of length $d$ modulo $p^k$ and $n\ge d$. Then we have that
\[
\frac{n}{pd} - 1\le\sigma\pa*{\step_{0},\dots,\step_{n-1}}.
\]
\end{cor}

\begin{proof}
Denote by $\ang*{\cdot,\cdot}_{R}$ the dot product
in the ring $R$. Clearly $\ang*{\u,\step_{i}}_{\Rmod p}\ne0$ implies that $\ang*{ \u,\step_{i}}_{\Rmod{p^{k}}}\ne0$.
Hence
\[
\den_{\u}\pa*{\step_{0}\bmod{p},\dots,\step_{n-1}\bmod{p}}\le\den_{\u}\pa*{\step_{0},\dots,\step_{n-1}}
\]
and so $\sigma\pa*{\step_{0}\bmod{p},\dots,\step_{n-1}\bmod{p}}\le\sigma\pa*{\step_{0},\dots,\step_{n-1}}$.
\end{proof}

\begin{proof}[Proof of Proposition~\ref{prop:coeff-prob}]
We expand $f\pa*{r+X}$ into two ways. By Taylor's expansion we have
\[
f\pa*{r+X}=f\pa*{r}+f'\pa*{r}X+\frac{1}{2}f''\pa*{r}X^{2}+\dots+X^{n}\text{.}
\]
On the other hand, we apply Newton's binomial theorem
\begin{align*}
f\pa*{r+X} & =\rC_{0}+\rC_{1}\pa*{r+X}+\dots+\rC_{n}\sum_{j=0}^{n}\binom{n}{j}r^{j}X^{n-j}\\
 & =\sum_{j=0}^{n}\rC_{j}r^{j}+\sum_{j=1}^{n}\rC_{j}\binom{j}{1}r^{j-1}X+\dots+X^{n}\text{.}
\end{align*}
Here $\rC_n = 1$.
Comparing the coefficients in both expansions we get that
\begin{equation}\label{eq:poly-taylor-coeff}
f^{\pa*{i}}\pa*{r}/i!=\sum_{j=0}^{n}\rC_{j}\binom{j}{i}r^{j-i}
\text{,}\qquad
i=0,\dots,d-1
\text{.}
\end{equation}
Since $p\nmid r$ we get that $f^{\pa*{i}}\pa*{r}/i!\equiv\gamma_{i}\pmod{p^{m_{i}}}$
if and only if $f^{\pa*{i}}\pa*{r}r^{i}/i!\equiv\gamma_{i}r^{i}\pmod{p^{m_{i}}}$.

Next we apply Proposition~\ref{prop:zero-random-walk} with $V=\pa*{\Rmod{p^{M}}}^{d}$ and $\step_{0},\dots,\step_{n}$ the Pascal vectors of length $d$ modulo $p^k$ (see equation \eqref{eq:pascal-vector}).
By equation \eqref{eq:poly-taylor-coeff}, we have 
\begin{equation}\label{eq:poly-taylor-coeff-vector}
\sum_{i=0}^{n-1}\rC_{i}\step_{i}r^{i}+\step_{n}r^{n}=\pa*{f\pa*{r},f'\pa*{r}r,\dots,\frac{1}{d!}f^{\pa*{d}}\pa*{r}r^{d}}\text{.}
\end{equation}
If we set $S=\set*{ \pa*{x_{0},\dots,x_{d-1}} \in V:\forall i,x_{i}\equiv r^{i}\gamma_{i}\pmod{p^{m_{i}}}} $,
then by Proposition~\ref{prop:zero-random-walk} we have that
\begin{multline*}
\Pr\pa*{\sum_{i=0}^{n-1}\rC_{i}\step_{i}r^{i}+\step_{n}r^{n}\in S}=\sum_{\u\in-\step_{n}r^{n}+S}\nu_{r}\pa*{\u}\\
=\frac{\#S}{\#V}+O\pa*{\#S\cdot\exp\pa*{-\frac{\tau\sigma\pa*{\step_{0},\step_{1},\dots,\step_{n-1}}}{p^{2M}}}},
\end{multline*}
where $\tau$ is as defined in equation \eqref{eq:tau}.
Since $\#S=\prod_{i=0}^{d-1}p^{M-m_{i}}=p^{dM}p^{-N}\le p^{dM}$, $\#V=p^{dM}$,
by Corollary~\ref{cor:minimal-den-mod-q}
\begin{equation*}
\Pr\pa*{\sum_{i=0}^{n}\rC_{i}\step_{i}r^{i}\in S}=p^{-N}+O\pa*{\exp\pa*{-\frac{\tau n}{p^{2M+1}d}+dM\cdot\log p}}.
\end{equation*}
We left only with showing that $\tau > 0$ which is true since $\rC_i \bmod p$ is non-constant.
\end{proof}
\section{Proof of the main theorem}
\label{sec:main-theorem}

We prove that for $f$ as in Theorem~\ref{thm:main}, $f\pa*{r+pX} \bmod{p^{k}}$ is uniformly distributed in $\Upsilon_k$ up to an exponentially small error.
In this section we use the notation $f_r\pa*{X} = f\pa*{r+pX}$ as in equation~\eqref{prop:root-count-split}.
\begin{lem}
\label{lem:ups-equidist}
Let $f$ be a random polynomial defined as in Theorem
\ref{thm:main} and let $0<\varepsilon<1$.
Then there exists $c>0$ depending only on $\varepsilon$ and the distribution of $\rC_i$,
such that for any integer $1\le r<p$, a positive integer $k\le \frac{\varepsilon\log n}{2\log p}$
and a fixed polynomial $h\in\Upsilon_{k}\subseteq\Rmod{p^{k}}\br*{X}$,
we have
\[
\Pr\pa*{f_{r}\equiv h\pmod{p^{k}}}=\frac{1}{\#\Upsilon_{k}}+O\pa*{\exp\pa*{-cn^{1-\varepsilon}}}
\]
as $n\to\infty$.
\end{lem}

\begin{proof}
Recall that
\[
f_{r}\pa*{X}=f\pa*{r}+f'\pa*{r}pX+\frac{1}{2}f''\pa*{r}p^{2}X^{2}+\dots+p^{n}X^{n}\text{.}
\]
As $h\in\Upsilon_{k}$, it is of the form $h\pa*{X}=c_{0}+c_{1}pX+\dots+c_{k-1}p^{k-1}X^{k-1}$.
We have
\[
f_{r}\equiv h\pmod{p^{k}}\iff
\frac{1}{i!}f^{\pa*{i}}\pa*{r}\equiv c_{i}\pmod{p^{k-i}}
\text{,}\quad i=0,\dots,k-1
\text{.}
\]
Apply Proposition~\ref{prop:coeff-prob} with
$d=k$, $m_i = k-i$ and $\gamma_{i} = c_i$ so that $N=\sum_{i=0}^{k-1}\pa*{k-i}=k\pa*{k+1}/2$, $M=k$ and
\[
\Pr\pa*{f_r \equiv h\pmod{p^{k}}}=p^{-k\pa*{k+1}/2}+O\pa*{\exp\pa*{-\frac{\tau n}{2p^{2k+1}k}+k^{2}\log p}}\text{.}
\]
The main term is indeed $1/\#\Upsilon_{k}$ by equation \eqref{eq:ups-size-1}.
By the assumption on $k$, the error term is $O\pa*{\exp\pa*{-cn^{1-\varepsilon}}}$ as needed.
\end{proof}
\begin{proof}[Proof of Theorem~\ref{thm:main}]
By equation~\eqref{prop:root-count-split} we have $\rCount f=\sum_{r=0}^{p-1}\rCount{f_{r}}$.
Thus
\begin{equation}\label{eq:main-split}
\Ex\br*{\rCount f}=\Ex\br*{\rCount{f_{0}}}+\sum_{r=1}^{p-1}\Ex\br*{\rCount{f_{r}}}\text{.}
\end{equation}
Let $k=\floor{\frac{\pa*{1-\dlt}\log n}{2\log p}}$ where $\dlt$ is a positive real to be defined later.
So for any $1\le r<p$, we apply the law of total expectation and
Lemma~\ref{lem:ups-equidist} to get $c_1>0$ such that
\begin{equation}
\label{eq:main-1}
\begin{aligned}
\Ex\br*{\rCount{f_{r}}}&=\sum_{h\in\Upsilon_{k}}\Ex\br*{\rCount{f_{r}}\cond f_{r}\equiv h\pmod{p^{k}}}\Pr\pa*{f_{r}\equiv h\pmod{p^{k}}}\\
&=\sum_{h\in\Upsilon_{k}}\Ex\br*{\rCount{f_{r}}\cond f_{r}\equiv h\pmod{p^{k}}}\pa*{\frac{1}{\#\Upsilon_{k}}+O\pa*{\exp\pa*{-c_{1}n^{\dlt}}}}\text{.}
\end{aligned}
\end{equation}
Since $\Ex\br*{\rCount{f_{r}}\cond f_{r}\equiv h\pmod{p^{k}}}=O\pa*{n}$
and $\#\Upsilon_{k}=p^{k\pa*{k+1}/2}$ (see equation~\eqref{eq:ups-size-1}),
we may bound the error term in equation \eqref{eq:main-1} as follows
\begin{align*}
\sum_{h\in\Upsilon_{k}}\Ex\br*{\rCount{f_{r}}\cond f_{r}\equiv h\pmod{p^{k}}}\exp\pa*{-c_{1}n^{\dlt}}
&=O\pa*{np^{k\pa*{k+1}/2}\exp\pa*{-c_{1}n^{\dlt}}}\\
&=O\pa*{\exp\pa*{-c_{1}n^{\dlt} + c_2\log^2 n}} \\
&=O\pa*{\exp\pa*{-c_{3}n^{\dlt}}}
\end{align*}
for some $c_{2}, c_{3}>0$. Plugging this in equation \eqref{eq:main-1} gives
\begin{equation}\label{eq:main-2}
\Ex\br*{\rCount{f_{r}}}=\frac{1}{\#\Upsilon_{k}}\sum_{h\in\Upsilon_{k}}\Ex\br*{\rCount{f_{r}}\cond f_{r}\equiv h\pmod{p^{k}}} +O\pa*{\exp\pa*{-c_{3}n^{\dlt}}}\text{.}
\end{equation}

Let $g$ be a random polynomial distributed according to the law
\[
\Pr\pa*{g\in E}=\frac{1}{\#\Upsilon_{k}}\sum_{h\in\Upsilon_{k}}\Pr\pa*{f_{r}\in E\cond f_{r}\equiv h\pmod{p^{k}}}
\text{,}\quad E\subseteq \ZZ_p\br*{X}\text{ Borel.}
\]
This distribution is well-defined for $n$ sufficiently large, since $\Pr\pa*{f_{r}\equiv h\pmod{p^{k}}}$ is bounded away from zero by Lemma~\ref{lem:ups-equidist}.

Then
\begin{equation}\label{eq:ex-g-1}
\Ex\br*{\rCount g}=\frac{1}{\#\Upsilon_{k}}\sum_{h\in\Upsilon_{k}}\Ex\br*{\rCount{f_{r}}\cond f_{r}\equiv h\pmod{p^{k}}}\text{.}
\end{equation}
On the other hand, $g\bmod{p^{k}}$ is distributed uniformly in $\Upsilon_{k}$.
Assume $n$ is sufficiently large with respect to $p$.
Then $k \ge {\log n}/{4\log p}$, so $\deg f_{r}=n=p^{4\cdot{\log n}/{4\log p}}\le p^{4k}$.
Hence $\deg g\le p^{4k}$ almost surely. So by Proposition
\ref{prop:ups-root-count-rephrased} we conclude that
\begin{equation}\label{eq:ex-g-2}
\Ex\br*{\rCount g}=\frac{1}{p+1}+O\pa*{p^{-\pa*{1-\dlt}k/2}}\text{.}
\end{equation}

Plugging equations \eqref{eq:ex-g-1} and \eqref{eq:ex-g-2} into equation \eqref{eq:main-2} gives
\begin{equation*}
\Ex\br*{\rCount{f_{r}}}=\frac{1}{p+1}+O\pa*{p^{-\pa*{1-\dlt}k/2}+\exp\pa*{-c_{2}n^{\dlt}}}\text{.}
\end{equation*}
We choose $\dlt$ such that $\dlt < 2 \varepsilon$, so that $p^{-\pa*{1-\dlt}k/2}=O\pa*{n^{-1/4+\varepsilon}}$. Thus
\begin{equation*}
\Ex\br*{\rCount{f_{r}}}=\frac{1}{p+1}+O\pa*{n^{-1/4+\varepsilon}}\text{.}
\end{equation*}
Finally, we finish the proof by substituting $\Ex\br*{\rCount{f_{r}}}$ into equation \eqref{eq:main-split}.
\end{proof}

\section{The expected value of \texorpdfstring{$\rCountF[\ZZ_p]{f_0}$}{C\_Zp(f\_0)}}
\label{sec:additional-results}

In this section we prove two results on $\Ex\br*{\rCountF[\ZZ_p]{f_0}}$
mentioned in the introduction, equation \eqref{eq:f0-bound} and Proposition~\ref{prop:better-root-count}.
\begin{proof}[Proof of equation \eqref{eq:f0-bound}]
We have that
\[
\Ex\br*{\rCount{f_{0}}}=\sum_{k=1}^{\infty}k\Pr\pa*{\rCount{f_{0}}=k}=\sum_{k=1}^{\infty}\Pr\pa*{\rCount{f_{0}}\ge k}\text{.}
\]
Since $\deg\pa*{f_{0}\bmod{p^{k}}}<k$, Proposition~\ref{prop:roots-in-zp} gives that
$\rCount{f_{0}}\ge k$ only if $f_0\equiv0\pmod{p^{k}}$. Put $q=\Pr\pa*{\rC_0 \equiv 0 \pmod{p}}$. Since $\rC_{0},\dots,\rC_{n-1}$ are i.i.d., we conclude that
\[
  \Pr\pa*{\forall i<k,\rC_{i}\equiv0\pmod{p}} = q^k \text{.}
\]
Hence
\begin{align*}
\Ex\br*{\rCount{f_{0}}} & \le\sum_{k=1}^{\infty}\Pr\pa*{f_0\equiv0\pmod{p^{k}}}\\
 & \le\sum_{k=1}^{\infty}\Pr\pa*{\forall i<k,\rC_{i}\equiv0\pmod{p}}
 =\sum_{k=1}^{\infty}q^{k} = \frac{q}{1-q}\text{.}\qedhere
\end{align*}
\end{proof}
Proposition~\ref{prop:better-root-count} follows from
\begin{lem}
\label{lem:no-non-zero-roots}
Assume $\rC_0, \dots,\rC_{n-1}$ satisfy the hypothesis of Proposition~\ref{prop:better-root-count}. The polynomial $f_{0}$ has no non-zero roots in $\ZZ_p$
almost surely.
\end{lem}

\begin{proof}
Assume $f_{0}$ has a non-zero root in $\ZZ_p$ and let $\alpha$ be such root.
We argue by induction that $\rC_{i}=0$ almost surely, for $i=0,\dots,n-1$.

For $i=0$, reduce the equation $f_{0}\pa*{\alpha}=0$ modulo $p$ to get:
\begin{equation*}
\rC_{0}\equiv\rC_{0}+\rC_{1}p\alpha+\dots+p^{n}\alpha^{n}\equiv0\pmod{p} \text{.}
\end{equation*}
Hence, $p\mid\rC_{0}$ and by the hypothesis we get
$\rC_{0}=0$ almost surely.

Next, assume that $\rC_{0}=\rC_{1}=\dots=\rC_{i-1}=0$ almost surely.
Since $\alpha\ne0$ there exists a non-negative integer $v \ge 0$ such that $p^v \mid \alpha$ and $p^{v+1} \nmid \alpha$.

We reduce the equation $f_{0}\pa*{\alpha}=0$ modulo $p^{vi+i+1}$ to get:
\begin{equation*}
\rC_{i}p^{i}\alpha^{i}+\rC_{i+1}p^{i+1}\alpha^{i+1}+\dots+p^{n}\alpha^{n} \equiv 0\pmod{p^{vi+i+1}}.
\end{equation*}
Write $\alpha = p^v \tilde\alpha$ where $p\nmid \tilde{\alpha}\in \ZZ_p$ so
\begin{equation*}
\rC_{i}p^{\pa*{v+1}i}\tilde{\alpha}^{i} +\rC_{i+1}p^{\pa*{v+1}\pa*{i+1}}\tilde{\alpha}^{i+1} + \dots + p^{\pa*{v+1}n}\tilde{\alpha}^{n} \equiv 0\pmod{p^{vi+i+1}}.
\end{equation*}
Since $\rC_j\tilde{\alpha} \in \ZZ_p$ and $p^{\pa*{v+1}j} \mid p^{vi+i+1}$ for any $j > i$ we get that
\begin{equation*}
\rC_{i}p^{vi+i}\tilde{\alpha}^{i} \equiv 0\pmod{p^{vi+i+1}}.
\end{equation*}
Thus $\rC_{i}\tilde{\alpha}^i\equiv0\pmod{p}$ and since $p\nmid\tilde{\alpha}$, we get that $p\mid \rC_{i}$.
By the hypothesis, $\rC_{i}=0$ almost surely, as needed.

This means that $f_{0}\pa*{X}=p^{n}X^{n}$ almost surely assuming
the event that $f_{0}$ has a non-zero root.
But clearly the only root of $p^{n}X^{n}$ is zero. This contradiction shows that $f_0$ has no non-zero roots in $\ZZ_p$ almost surely.
\end{proof}

\newpage
\appendix
\section{Roots of Haar random polynomials}
\label{app:haar-random-polynomial}
Consider the random polynomial
\[
f\pa*{X}=\rC_{0}+\rC_{1}X+\dots+\rC_{n-1}X^{n-1}+X^{n}
\]
where $\rC_{0},\dots,\rC_{n-1}$ are i.i.d.\ random variables which  take values in the ring $\ZZ_{p}$ according to Haar measure on $\ZZ_p$.
We denote by $\rCount f$ the number of roots of $f$ in $\ZZ_p$ without multiplicities,
i.e.,
\[
\rCount f=\#\set*{ \alpha\in\ZZ_{p}:f\pa*{\alpha}=0} \text{.}
\]
We prove the following formula:
\begin{equation}\label{eq:2nd-haar-roots}
\Ex\br*{\rCount f} = \frac{p}{p+1}.
\end{equation}
Moreover, we prove another formula:
\begin{equation}\label{eq:haar-roots-modulo-r}
\Ex\br*{\rCount{f_r}} = \frac{1}{p+1} \quad (\forall r=0,\dots,p-1),
\end{equation}
where $f_r$ is the polynomial $f_r\pa*{X} = f\pa*{r+pX}$.

We recall the definition of Haar measure. The $p$-adic norm induces a
metric on $\QQ_p$ defined
by $d\pa*{\alpha,\beta} = \abs*{\alpha-\beta}_{p}$.
The open balls of this metric are of the form $\alpha+p^{k}\ZZ_{p}$
for some $\alpha\in\QQ_{p}$ and $k\in\ZZ$.
Since the
$p$-adic absolute value is discrete, every open ball is also closed
and compact.
By Haar's theorem
(see \cite[Chapter XI]{halmos_measure_1974}), there exists a unique up to a constant, regular measure
$\mu$ on Borel subsets of $\QQ_{p}$ such that for any Borel set
$E\subseteq\QQ_{p}$ and $\alpha\in\QQ_{p}$:
\begin{align*}
\mu\pa*{\alpha+E} & =\mu\pa*{E}\\
\mu\pa*{\alpha E} & =\abs*{\alpha}_{p}\mu\pa*{E}
\end{align*}
Such a measure is called a \emph{Haar measure}.

We denote by $\mu$ the Haar measure on $\QQ_p$ such that $\mu\pa*{\ZZ_p} = 1$, so the law of $\rC_i$ is $\mu$ restricted to $\ZZ_p$.
All integrals in this appendix are Lebesgue integrals according to
the measure $\mu$ or to a product measure of its copies.

We start with surveying the tools we shall use to derive equation~\eqref{eq:2nd-haar-roots}. We start with the Igusa's local
zeta functions (for more details see \cite{denef_report_1990}).
Let $F\in\ZZ_{p}\big[\xv\big]$ be a multivariate polynomial in
$\xv=\pa*{X_{1},X_{2},\dots,X_{m}}$. We define the \emph{Igusa's
local zeta function associated to $F$} to be
\[
Z_{F}\pa*{s}=\int_{\ZZ_{p}^{m}}\abs*{F\pa*{\alv}}_{p}^{s}\diff\alv\text{,}
\]
for $s\in\CC$, $\Re\pa*{s}>0$.

We also associate to $F$ \emph{the Poincar\'e series}:
Let $N_k$ be the number of solutions of $F\pa*{\xv}\equiv 0 \pmod{p^k}$.
Then
\begin{equation*}
P_{F}\pa*{t}=\sum_{k=0}^{\infty}p^{-mk}N_{k}t^{k}\text{,}
\end{equation*}
for $t\in\CC$, $\abs*{t} < 1$

We have a nice formula relating $Z_F$ and $P_F$:
\begin{equation}\label{eq:zeta-poincare}
P_{F}\pa*{p^{-s}}=\frac{1-p^{-s}Z_{F}\pa*{s}}{1-p^{-s}}\text{,}
\end{equation}
see \cite[Section 1.2]{denef_report_1990}.

The next tool is an integration formula, see \cite[Proposition 2.3]{evans_expected_2006}:
Let $f\in\QQ_{p}\br*{X}$ be a polynomial and
let $g:\QQ_{p}\to\RR^{+}$ be a measurable function. Then
\begin{equation}\label{eq:padic-integral}
\int_{\ZZ_{p}}g\circ f\pa*{\alpha}\abs*{f'\pa*{\alpha}}_p\diff\alpha=\int_{\QQ_{p}}g\pa*{\beta}\rCount{f-\beta}\diff\beta\text{.}
\end{equation}

We are now ready to prove our formulas:
\begin{proof}[Proof of equation~\eqref{eq:2nd-haar-roots}]
The random variable $\rC_{0}-\beta$ distributes the same as $\rC_{0}$
for any $\beta\in\ZZ_{p}$, since Haar measure is invariant under translations.
Hence $\Ex\br*{\rCount{f-\beta}}$ is a constant that is independent
of $\beta$. By Fubini's theorem:
\begin{equation}\label{eq:haar-integral-1st-fubini}
\Ex\br*{\rCount f}=\int_{\ZZ_{p}}\Ex\br*{\rCount{f-\beta}}\diff\beta=\Ex\br*{\int_{\ZZ_{p}}\rCount{f-\beta}\diff\beta}\text{.}
\end{equation}
If $\beta\in\QQ_{p}\setminus\ZZ_{p}$ then $f-\beta$ has no
roots in $\ZZ_{p}$ and so $\rCount{f-\beta}=0$. Hence, by equation~\eqref{eq:padic-integral} with $g=1$ we get:
\[
\int_{\ZZ_{p}}\rCount{f-\beta}\diff\beta=\int_{\QQ_{p}}\rCount{f-\beta}\diff\beta=\int_{\ZZ_{p}}\abs*{f'\pa*{\alpha}}_p\diff\alpha
\text{.}
\]
Plugging this into equation \eqref{eq:haar-integral-1st-fubini} and using Fubini's theorem again gives
\begin{equation}\label{eq:haar-integral-2st-fubini}
\Ex\br*{\rCount f}=\int_{\ZZ_{p}}\Ex\br*{\abs*{f'\pa*{\alpha}}_p} \diff\alpha
\end{equation}

To calculate $\Ex\br*{\abs*{f'\pa*{\alpha}}_p}$
we define a multivariate polynomial $F_{\alpha}\in\ZZ_{p}\br*{X_{0},\dots,X_{n-1}}$:
\[
F_{\alpha}\pa*{\xv}=X_{1}+\dots+\pa*{n-1}\alpha^{n-2}X_{n-1}+n\alpha^{n-1}\text{,}
\]
so that $f'\pa*{\alpha}=F_{\alpha}\pa*{\rC_{0},\dots,\rC_{n-1}}$.
Put $\rCv=\pa*{\rC_{0},\dots,\rC_{n-1}}$ to get that
\begin{equation}\label{eq:expected-value-as-zeta}
\Ex\br*{\abs*{f'\pa*{\alpha}}_p}=\int_{\ZZ_{p}^{n}}\abs*{F_{\alpha}\pa*{\rCv}}_p\diff\rCv=Z_{F_{\alpha}}\pa*{1}.
\end{equation}

Next we compute the Poincaré series of $F_{\alpha}$.
Since we can isolate $X_1$ in the equation $F_{\alpha}\big(\xv\big)\equiv0\pmod{p^{k}}$ there are $p^{k\pa*{n-1}}$
solutions modulo $p^k$.
Hence
\[
P_{F_{\alpha}}\pa*{t}=\sum_{n=0}^{\infty}p^{-k}t^{k}=\frac{1}{1-p^{-1}t}\text{.}
\]
By equation~\eqref{eq:zeta-poincare} we get
\[
\frac{1}{1-p^{-1}p^{-s}}=\frac{1-p^{-s}Z_{F_{\alpha}}\pa*{s}}{1-p^{-s}}.
\]
So $Z_{F_\alpha}\pa*{1}=p/\pa*{p+1}$. Taking equation~\eqref{eq:expected-value-as-zeta} and equation~\eqref{eq:haar-integral-2st-fubini} into account we get that
\[
\Ex\br*{\rCount f}
=\int_{\ZZ_{p}}Z_{F_\alpha}\pa*{1}\diff\alpha
=\frac{p}{p + 1}
\text{.}
\qedhere
\]
\end{proof}
\begin{proof}[Proof of equation~\eqref{eq:haar-roots-modulo-r}]
By grouping the roots according to their value modulo $p$ we get
\begin{equation*}
	\rCount f = \sum_{r=0}^{p-1}\rCount{f_{r}}.
\end{equation*}
By linearity of expectation and equation \eqref{eq:2nd-haar-roots} we get
\begin{equation*}
	\sum_{r=0}^{p-1} \Ex\br*{\rCount{f_{r}}} = \frac{p}{p+1}.
\end{equation*}
Therefore, it suffices to show that
\begin{equation*}
\Ex\br*{\rCount{f_{0}}} = \Ex\br*{\rCount{f_{1}}} = \dots = \Ex\br*{\rCount{f_{p-1}}}.
\end{equation*}

We take a look at the polynomial $f\pa*{r+X}$.
The coefficients of $f\pa*{r+X}$ are achieved by multiplying the coefficients of $f$ with a unipotent matrix.
Hence, the coefficients of $f\pa*{r+X}$ remains i.i.d.\ and distributed according to $\mu$ restricted to $\ZZ_p$.
Thus, the law of $f_r\pa*{X} = f\pa*{r+pX}$ is the same for all $r=0,\dots,p-1$, and $\Ex\br*{\rCount{f_0}} = \Ex\br*{\rCount{f_r}}$.
\end{proof}

\newpage
\bibliographystyle{alpha}
\bibliography{bib/main}

\begin{thebibliography}{BGMR06}

\bibitem[Bac64]{bachman_introduction_1964}
George Bachman.
\newblock {\em Introduction to {$p$}-adic numbers and valuation theory}.
\newblock Academic Press, New York-London, 1964.

\bibitem[BCFG21]{bhargava_density_2021}
Manjul Bhargava, John Cremona, Tom Fisher, and Stevan Gajovi{\'c}.
\newblock The density of polynomials of degree $n$ over $\mathbb{Z}_p$ having
  exactly $r$ roots in $\mathbb{Q}_p$.
\newblock {\em arXiv preprint arXiv:2101.09590}, 2021.

\bibitem[BGMR06]{buhler_probability_2006}
Joe Buhler, Daniel Goldstein, David Moews, and Joel Rosenberg.
\newblock The probability that a random monic {$p$}-adic polynomial splits.
\newblock {\em Exp. Math.}, 15(1):21--32, 2006.

\bibitem[BP31]{bloch_roots_1932}
A.~Bloch and G.~P{\'o}lya.
\newblock On the {R}oots of {C}ertain {A}lgebraic {E}quations.
\newblock {\em Proc. Lond. Math. Soc. (2)}, 33(2):102--114, 1931.

\bibitem[BV19]{breuillard_irreducibility_2019}
Emmanuel Breuillard and {P\'e}ter~P. Varj{\'u}.
\newblock Irreducibility of random polynomials of large degree.
\newblock {\em Acta Math.}, 223(2):195--249, 2019.

\bibitem[Car18]{caruso_zeros_2018}
Xavier Caruso.
\newblock Where are the zeroes of a random $p$-adic polynomial?
\newblock Unpublished notes. Available at
  \url{http://xavier.toonywood.org/papers/publis/randompoly-talk.pdf}, 2018.

\bibitem[CDG87]{chung_random_1987}
F.~R.~K. Chung, Persi Diaconis, and R.~L. Graham.
\newblock Random walks arising in random number generation.
\newblock {\em Ann. Probab.}, 15(3):1148--1165, 1987.

\bibitem[Con]{conrad_hensel_xxxx}
Keith Conrad.
\newblock {H}ensel's lemma.
\newblock Unpublished notes. Available at
  \url{https://kconrad.math.uconn.edu/blurbs/gradnumthy/hensel.pdf}.

\bibitem[Den91]{denef_report_1990}
Jan Denef.
\newblock Report on {I}gusa's local zeta function.
\newblock Number 201-203, pages Exp. No. 741, 359--386 (1992). 1991.
\newblock S\'{e}minaire Bourbaki, Vol. 1990/91.

\bibitem[Eis95]{eisenbud_commutative_2013}
David Eisenbud.
\newblock {\em Commutative algebra}, volume 150 of {\em Graduate Texts in
  Mathematics}.
\newblock Springer-Verlag, New York, 1995.
\newblock With a view toward algebraic geometry.

\bibitem[EO56]{erdos_number_1956}
Paul Erd{\"o}s and A.~C. Offord.
\newblock On the number of real roots of a random algebraic equation.
\newblock {\em Proc. Lond. Math. Soc. (3)}, 6:139--160, 1956.

\bibitem[Eva06]{evans_expected_2006}
Steven~N. Evans.
\newblock The expected number of zeros of a random system of {$p$}-adic
  polynomials.
\newblock {\em Electron. Commun. Probab.}, 11:278--290, 2006.

\bibitem[Fin47]{fine_binomial_1947}
N.~J. Fine.
\newblock Binomial coefficients modulo a prime.
\newblock {\em Amer. Math. Monthly}, 54:589--592, 1947.

\bibitem[Hal50]{halmos_measure_1974}
Paul~R. Halmos.
\newblock {\em Measure {T}heory}.
\newblock D. Van Nostrand Company, Inc., New York, N. Y., 1950.

\bibitem[IM71]{ibragimov_expected_1971}
I.~A. Ibragimov and N.~B. Maslova.
\newblock The mean number of real zeros of random polynomials. {I}.
  {C}oefficients with zero mean.
\newblock {\em Teor. Verojatnost. i Primenen.}, 16:229--248, 1971.

\bibitem[Kac43]{kac_average_1943}
M.~Kac.
\newblock On the average number of real roots of a random algebraic equation.
\newblock {\em Bull. Amer. Math. Soc.}, 49:314--320, 1943.

\bibitem[KL21]{kulkarni_padic_2021}
Avinash Kulkarni and Antonio Lerario.
\newblock {$p$}-adic integral geometry.
\newblock {\em SIAM J. Appl. Algebra Geom.}, 5(1):28--59, 2021.

\bibitem[Lan94]{lang_algebraic_1970}
Serge Lang.
\newblock {\em Algebraic number theory}, volume 110 of {\em Graduate Texts in
  Mathematics}.
\newblock Springer-Verlag, New York, second edition, 1994.

\bibitem[LO38]{littlewood_number_1938}
J.~E. Littlewood and A.~C. Offord.
\newblock On the {N}umber of {R}eal {R}oots of a {R}andom {A}lgebraic
  {E}quation.
\newblock {\em J. Lond. Math. Soc.}, 13(4):288--295, 1938.

\bibitem[ML20]{manssour_probabilistic_2020}
Rida Ait~El Manssour and Antonio Lerario.
\newblock Probabilistic enumerative geometry over $ p $-adic numbers: linear
  spaces on complete intersections.
\newblock {\em arXiv preprint arXiv:2011.07558}, 2020.

\bibitem[OP93]{odlyzko_zeros_1993}
A.~M. Odlyzko and B.~Poonen.
\newblock Zeros of polynomials with {$0,1$} coefficients.
\newblock {\em Enseign. Math. (2)}, 39(3-4):317--348, 1993.

\bibitem[Rog61]{rogozin_increase_1961}
B.~A. Rogozin.
\newblock On the increase of dispersion of sums of independent random
  variables.
\newblock {\em Teor. Verojatnost. i Primenen}, 6:106--108, 1961.

\bibitem[S{\"o}z17a]{soze_real_2017a}
Ken S{\"o}ze.
\newblock Real zeroes of random polynomials, {I}. {F}lip-invariance,
  {T}ur\'{a}n's lemma, and the {N}ewton-{H}adamard polygon.
\newblock {\em Israel J. Math.}, 220(2):817--836, 2017.

\bibitem[S{\"o}z17b]{soze_real_2017b}
Ken S{\"o}ze.
\newblock Real zeroes of random polynomials, {II}. {D}escartes' rule of signs
  and anti-concentration on the symmetric group.
\newblock {\em Israel J. Math.}, 220(2):837--872, 2017.

\end{thebibliography}

\end{document}